\documentclass[12pt]{amsart}
\usepackage[active]{srcltx}
\usepackage[centertags]{amsmath}
\usepackage{latexsym}
\usepackage{amssymb}

\usepackage{color}

\DeclareMathOperator{\Gal}{\rm Gal}

\newcommand{\F}{\mathbb F}
\newcommand{\barF}{\overline{\mathbb F}}
\newcommand{\barFq}{\overline{\mathbb F}_q}
\newcommand{\Proj}{\mathbb P}

\newcommand{\GL}{\mathrm{GL}}
\newcommand{\PGL}{\mathrm{PGL}}

\newcommand{\hide}[1]{}
\newtheorem{dummy}{Dummy}

\newtheorem{lemma}[dummy]{Lemma}

\newtheorem{theorem}[dummy]{Theorem}

\newtheorem{cor}[dummy]{Corollary}

\theoremstyle{definition}

\theoremstyle{remark}

\newtheorem{rem}[dummy]{Remark}

\begin{document}

\bibliographystyle{amsalpha}
\author{Sandro Mattarei}
\address{School of Mathematics and Physics\\
  University of Lincoln\\
  Brayford Pool\\
  Lincoln, LN6 7TS\\
  United Kingdom}
\email{smattarei@lincoln.ac.uk}
\author{Marco Pizzato}
\address{Dipartimento di Matematica\\
Universit\`a degli Studi di Trento\\
via Sommarive 14\\
I-38123 Povo (Trento)\\
Italy}
\email{marco.pizzato1@gmail.com}

\title{Cubic rational expressions over a finite field}

\begin{abstract}
We study and partially classify cubic rational expressions $g(x)/h(x)$ over a finite field $\F_q$,
up to pre- and post-composition with independent M\"obius transformations.
In particular, we obtain a full classification when $q$ is even,
and prove an upper bound of $4q$ for the number of equivalence classes when $q$ is odd.
\end{abstract}


\maketitle

\section{Introduction}\label{sec:intro}

The initial motivation for this study arose from an enumeration problem for irreducible polynomials over a finite field satisfying certain symmetries.
Gauss gave a formula for the number of all irreducible monic polynomials of a given degree over a field $\F_q$.
A similar formula counting the self-reciprocal irreducible monic polynomials of degree $2n$ was found by Carlitz in~\cite{Carlitz:srim}.
Here a polynomial $F(x)$ over a field is {\em self-reciprocal} if $x^{\deg F}\cdot F(1/x)=F(x)$.
It is easy to see that any irreducible self-reciprocal polynomial of degree greater than one has even degree $2n$, and that it can actually be written in the form
$F(x)=x^n\cdot f(x+1/x)$, for some polynomial $f(x)$ of degree $n$.
This led Ahmadi, in~\cite{Ahmadi:Carlitz}, to study polynomials
obtained from such $f(x)$ through a more general {\em quadratic transformation,} namely, polynomials of the form
$F(x)=h(x)^n\cdot f\bigl(g(x)/h(x)\bigr)$,
where $g(x)$ and $h(x)$ are coprime polynomials with $\max(\deg g,\deg h)=2$, and $n=\deg f$.
Over $\F_q$ with $q$ odd, the enumeration formula found in~\cite[Theorem~2]{Ahmadi:Carlitz} for
the number of monic irreducible polynomials among those, of degree $2n>2$ and for a given $g(x)/h(x)$,
equals Carlitz's count for the special case where $g(x)/h(x)=(x^2+1)/x=x+x^{-1}$.

The fact that Carlitz's formula extends unchanged to arbitrary quadratic transformations, at least for $q$ odd,
was given a simple explanation in~\cite{MatPiz:self-reciprocal}.
It depends on the fact that such enumeration is essentially unaffected
when the quadratic rational expression $g(x)/h(x)$ is replaced with any other obtained from it by composition, on both sides,
with independent M\"obius transformations (meaning rational expressions of degree $1$).
We call such rational expressions {\em equivalent} in this paper.
It is easy to see that all quadratic rational expressions $g(x)/h(x)$ over an algebraically closed field of odd characteristic are equivalent.
Over a finite field $\F_q$ of odd characteristic they split into two equivalence classes
(see Section~\ref{sec:quadratic}).
However, the expressions in each of both classes happen to give the same enumeration formula for the irreducible polynomials $F(x)$ arising through them.

A natural extension is a study of irreducible polynomials $F(x)$ arising through a transformation of higher degree,
that is, obtained as described above but with $g(x)/h(x)$ a rational expression (also commonly called {\em rational function}) of arbitrary degree $r$.
The special case where $g(x)=1$ is of practical relevance because of a direct connection with {\em transformation shift registers,}
which have numerous applications including in the design of stream ciphers,
see~\cite[Section~3]{Cohen+:TSR}.
Polynomials arising through a more general cubic transformation were addressed in~\cite{MatPiz:cubic-transformation},
where various enumeration formulas were found.

For such investigations it is desirable to have a classification up to equivalence, or at least a partial classification,
of cubic rational expressions (or {\em cubic expressions} for short) over a finite field.
That is the main goal of this paper.
We should mention that the study of rational expressions up to equivalence may be interpreted in terms of certain extensions of function fields.
At the end of Section~\ref{sec:equivalence} we review some literature on cubic extensions of function fields.
We discuss its connections with our work, and point out how the available literature falls short
of providing the explicit results that we need for applications such as those in~\cite{MatPiz:cubic-transformation}.
In the present paper we take a more direct approach and bypass the interpretation in terms of function fields.

As a preliminary step to understanding cubic expressions over a finite field, in Section~\ref{sec:cubic}
we describe a classification of cubic expressions over an algebraically closed field $K$.
This is not a hard task, but we could not locate the desired result in the literature.
Because equivalence classes are the orbits of a certain action of the group
$\PGL_2(K)\times\PGL_2(K)$
on the $7$-dimensional variety of all cubic expressions,
dimension reasons show that they must involve at least one parametric family.
In fact, according to Theorem~\ref{thm:cubic_expr_alg_closed}, every cubic expression over an algebraically closed  field
of characteristic different from two or three is either equivalent to $x^3$, or belongs to a certain $1$-parameter family.
The elements of the latter are not pairwise inequivalent, but one can explicitly tell which values of the parameter give rise to equivalent expressions,
in terms of a cross-ratio formed from the (up to four) ramification points of each expression.
We supplement that classification with corresponding classifications in the small characteristics three and two.


Passing from an algebraically closed ground field to a classification of cubic expressions over a finite field $\F_q$, which we address in Section~\ref{sec:cubic_finite},
requires understanding how the stabilizer in $\PGL_2(\F_q)\times\PGL_2(\F_q)$ of a cubic expression may permute its ramification points
(in $\barFq\cup\{\infty\}$).
While descending from the algebraic closure $\barFq$ to $\F_q$
could be framed in terms of a certain Galois cohomology group, that description would provide no saving in calculations
over the direct approach adopted here.
Theorem~\ref{thm:cubic_expr_finite} classifies the cubic expressions, over any finite field of characteristic at least five,
that have at most three (distinct) ramification points.
Theorem~\ref{thm:cubic_expr_finite_3} does the same in characteristic three.

In Section~\ref{sec:small_char} we obtain a corresponding result in characteristic two, which is Theorem~\ref{thm:cubic_expr_finite_2}.
Although that statement is more complicated than those in the odd characteristics,
it classifies the totality of cubic expressions over a finite field of characteristic two,
because such expressions cannot have more than two ramification points.
In particular, over a finite field $\F_q$ of characteristic two there are a total of $2q+2$ equivalence classes if $q$ is a square, and $2q$ otherwise.
We also compute the cardinalities of each of those equivalence classes (that is, the orbit lengths in the group action described above),
in Theorem~\ref{thm:cubic_expr_finite_2_count}.

Because of our restriction to cubic rational expressions with at most three ramification points,
we are unable to compute the exact number of equivalence classes over $\F_q$ in odd characteristic.
However, in Section~\ref{sec:bound}
we produce an upper bound of $4q$ for the number of classes over a finite field $\F_q$ of odd characteristic.
This requires an analysis of certain stabilizers in the action of $\PGL_2(\F_q)\times\PGL_2(\F_q)$.

After setting up terminology and preliminary results in Section~\ref{sec:equivalence}, we devote Section~\ref{sec:quadratic} to
a treatment of quadratic rational expressions over a finite field, which differs from the more direct one given in~\cite{MatPiz:self-reciprocal}.
The present version emphasizes the use of ramification and branch points, and the action of  $\PGL_2(\F_q)\times\PGL_2(\F_q)$,
anticipating in a simpler setting their later use for the case of cubic rational expressions.
In particular, it includes computing the cardinalities of the two equivalence classes of quadratic rational expressions over a finite field.


The research leading to this paper began when the second author was a PhD student at the University of Trento, Italy,
under the supervision of the first author.
Some of the results in this paper have appeared among other results in~\cite{Pizzato:thesis}.

%

\section{Preliminaries on rational expressions}\label{sec:equivalence}

Let $K$ be an arbitrary field for now.
Consider a rational expression $R(x)=g(x)/h(x)\in K(x)$, where $g(x)$ are $h(x)$ coprime
polynomials in $K[x]$.
Its {\em degree} $\deg R$ is the integer $\max(\deg g,\deg h)$;
the expression will be called {\em linear, quadratic, cubic,}
etc., when its degree equals $1,2,3,$ etc.
It is easy to see that the degree of the composite of two rational expressions
equals the product of their degrees.

In particular, the linear rational expressions
(which also go by various other names, such as {\em fractional linear transformations,} or {\em M\"obius transformations}) form a group,
which is isomorphic to the projective general linear group $\PGL_2(K)$.
Explicitly, to linear rational expressions
$A(x)=(a_{11}x+a_{12})/(a_{21}x+a_{22})$
and
$B(x)=(b_{11}x+b_{12})/(b_{21}x+b_{22})$
there correspond matrices $(a_{ij}),(b_{ij})\in\GL_2(K)$, unique up to multiplication by scalar factors,
whose product $(b_{ij})\cdot (a_{ij})$
is a matrix for the composite expression
$(B\circ A)(x)=B\bigl(A(x)\bigr)$.
The group of linear rational expressions in $K(x)$ is usually called {\em the M\"obius group} over $K$,
and for simplicity we will identify $\PGL_2(K)$ with it as above.
Thus, the M\"obius group is the group of projectivities of the projective line $\Proj^1(K)$,
which we will safely often identify with $K\cup\{\infty\}$.

The M\"obius group $\PGL_2(K)$ acts sharply 3-transitively on $\Proj^1(K)$, meaning that any ordered triple of distinct points of $\Proj^1(K)$
is taken to any other such triple by precisely one element of $\PGL_2(K)$.
In an explicit form which may be useful for calculations,
the unique M\"obius transformation which sends $\infty,0,1$ to three distinct elements $a,b,c$ of $K$, respectively, is
\[
\frac{a(b-c)x+b(c-a)}{(b-c)x+(c-a)}.
\]
This naturally extends to cases where one of $a,b,c$ equals $\infty$.

The M\"obius group is also the full group of automorphisms of the field extension $K(x)/K$,
where it acts by pre-composition (or substitution, or composition on the right).
However, in this paper we are interested in an action where a rational expression can be composed with
independent linear rational expressions on both sides.
Thus, if $R(x)\in K(x)$ is a rational expression of degree $n$, we let a pair $(B,A) \in \PGL_2(K) \times \PGL_2(K)$
act on $R(x)$ by setting $(B,A)\cdot R(x)=B(R(A^{-1}(x)))$.
This defines an action of the group $G(K)=\PGL(2,K)\times\PGL(2,K)$ on $K(x)$,
and we call two rational expressions {\em equivalent} (over $K$) if they belong to the same orbit.

Our main goal will be finding (some of) the equivalence classes (or $G(K)$-orbits)
on cubic rational expressions when $K$ is a finite field $\F_q$.
The following result shows, in particular, that the total number of cubic rational expressions over $\F_q$ equals $q^5(q^2-1)$.

\begin{lemma}\label{lemma:coprime_probability}
The number of distinct rational expressions of degree $r$ over $\F_q$ equals $q^{2r-1}(q^2-1)$.
\end{lemma}

\begin{proof}
The desired number equals $q-1$ times the number of (ordered) pairs of coprime monic polynomials $g(x)$ and $h(x)$ in $\F_q[x]$ with $\max(\deg g,\deg h)=r$,
so we need to show that those pairs are in number of $q^{2r-1}(q+1)$.

As a special of~\cite[Corollary~5]{Benjamin-Bennett}, the probability that a pair of monic polynomials in $\F_q[x]$,
of given positive degrees $r$ and $s$, are coprime equals $1-1/q$, and so the number of such pairs equals $q^{r+s-1}(q-1)$.
By induction, the number of pairs of coprime monic polynomials in $\F_q[x]$, with the former having degree $r$ and the latter having degree strictly inferior to $r$,
equals $q^{2r-1}$.
Splitting the set of pairs of coprime monic polynomials which we are counting according as to whether $\deg g=\deg h=r$, or $\deg g<\deg h=r$,
or $\deg g=r>\deg h$, we find altogether
$q^{2r-1}(q-1)+q^{2r-1}+q^{2r-1}=q^{2r-1}(q+1)$ distinct pairs, as desired.
\end{proof}

Before studying the $G(K)$-orbits on $K(x)$ with $K$ a finite field it will be necessary to determine them when $K$ is an algebraically closed field.
Characteristic two and three will require special treatment, due to possible inseparability but not just that.
A nonconstant rational expression $R(x)\in K(x)$ is called {\em separable} if the field extension
$K(x)/\bigl(K(R(x))\bigr)$ is separable, and {\em inseparable} otherwise.
We will only need a modicum of terminology from function field theory, pertaining to {\em ramification points} and the corresponding {\em ramification indices,}
and the following very special case of Hurwitz's Theorem (or the Riemann-Hurwitz formula),
which can be conveniently extracted from~\cite[Corollary~3.5.6]{Stichtenoth:book}
or~\cite[Theorem~7.16]{Rosen:number_theory_book}.

\begin{theorem}[Special case of Hurwitz's Theorem]\label{thm:Hurwitz}
Let $K$ be an algebraically closed field, and let
$R(x)\in K(x)$ be a nonconstant separable rational expression. Then
\[
2\deg(R) - 2 \geq \sum_{P \in  \Proj^1(K)} (e_R(P) -1),
\]
where $e_R(P)$ is the ramification index of $R(x)$ at $P$.
Equality holds if and only if either the characteristic of $K$ is zero or it does not divide
$e_R(P)$ for any $P \in \Proj^1(K)$.
\end{theorem}

In particular, according to Theorem~\ref{thm:Hurwitz} a separable cubic rational expression can have at most four ramification points $P$,
as the left-hand side of the inequality equals four in this case.
Of course this can also be seen directly from the fact that the finite ramification points of $R(x)=g(x)/h(x)$ are the zeroes of its derivative
$R'(x)=\bigl(g'(x)h(x)-g(x)h'(x)\bigr)/\bigl(h(x)\bigr)^2$, whose numerator has degree at most four
(and further considerations in case $\infty$ is a ramification point).

A classification of the $G(K)$-orbits on cubic rational expressions when $K$ is algebraically closed,
which we present in Section~\ref{sec:cubic}, will require using the action of $G(K)$ to suitably rearrange
the ramification points of an expression and their images, which are the corresponding {\em branch points}.
Note that distinct ramification points of a cubic expression necessarily map to distinct branch points.
A classification over a finite field $\F_q$, which we begin in Section~\ref{sec:cubic_finite},
will require studying the ramification points in its algebraic closure $\barF_q$ first.

We conclude this section by mentioning a natural field-theoretic interpretation of our results,
and how that relates to existing literature.
Classifying cubic rational expressions over a field $K$, up to the equivalence considered in this paper,
amounts to classifying the subfields $F$ of the field $K(x)$ of rational expressions, with $K\subseteq F$ and $|K(x)/F|=3$,
up to the action of $\Gal\bigl(K(x)/K\bigr)$.
This is because according to L\"uroth's theorem each such subfield has the form $F=K(z)$ for some cubic rational expression $z=R(x)$ in $K(x)$.
Also, $z$ is transcendental over $K$, and all elements of each of $\Gal\bigl(K(x)/K\bigr)$ and $\Gal\bigl(K(z)/K\bigr)$ are given by M\"obius transformations.

If we switch focus to the subfield $K(z)$, a rational function field, then here we are interested in certain cubic field extensions of that.
Cubic extensions $L/F$ of an arbitrary field $F$ of characteristic $p\ge 0$ were described in~\cite{Marques-Ward:primer},
extending classical results based on applications of Kummer theory and Artin-Schreier theory.
Any such extension has a generating polynomial of one of the forms $y^3-a$, $y^3-3y-a$ if $p\neq 3$, $y^3+ay+a^2$ if $p=3$,
for some $a\in F$.
Procedures were given in~\cite{Marques-Ward:primer} to decide when two such cubic field extensions are isomorphic.
The cases where $F$ is either a number field or a function field are of special interest,
and such field extensions have been extensively studied and tabulated over the years.

In particular, construction and tabulation of (geometric) dihedral extension of odd prime degree $\ell$ of a rational function field $F=K(z)$
was the primary goal of~\cite{Weir-Scheidler-Howe,Weir:thesis}, where $K$ is a perfect field of characteristic not dividing $2\ell$.
When $\ell=3$ a different approach was adopted in~\cite{Karemaker-Marques-Sijsling},
which focused on classifying such extensions with a given set of ramified places,
and including the case of characteristic two (but not three).

Besides a classification up to isomorphism of field extensions $L/F$, where $F=K(z)$ is a rational function field,
a classification up to {\em bi-isomorphism} was also obtained in~\cite{Karemaker-Marques-Sijsling},
which means allowing also automorphisms of $K(z)$ (hence M\"obius transformations over $K$).
When $L$ is also a rational function field, bi-isomorphism of function field extensions $L/F$ as defined in~\cite{Karemaker-Marques-Sijsling}
is thus closely related to equivalence of rational expressions as considered in this paper.
Consequently, some of our classification results, such as Theorem~\ref{thm:cubic_expr_finite},
match some of those obtained in~\cite[Section~5]{Karemaker-Marques-Sijsling}.
However, in~\cite{Karemaker-Marques-Sijsling} those extensions were discussed in terms of a generating polynomial,
of the general form $y^3-a$ or $y^3-3y-a$ for certain $a\in F=K(z)$,
while our focus on cubic rational expressions $R(x)$ is rather different and justifies a separate treatment.
Furthermore, our treatment includes the case of characteristic three, and a complete analysis of the case of characteristic two,
in Section~\ref{sec:small_char}.
Finally, and differently from~\cite{Karemaker-Marques-Sijsling},
our approach naturally leads to and includes counting results, both on the number of equivalence classes in Section~\ref{sec:bound},
and on cardinalities of individual equivalence classes in Section~\ref{sec:cubic_finite},

We are grateful to a referee of a previous version of this paper for calling our attention to some of the above references.

\section{Classification of quadratic expressions, revisited}\label{sec:quadratic}

Before we proceed towards a classification of cubic expressions over an algebraically closed field in the next section,
we review the much easier case of quadratic expressions.
Their equivalence classes can be determined by direct methods over an arbitrary field $K$,
as in~\cite[Theorem~2]{MatPiz:self-reciprocal} and the discussion which follows it,
but here we show how the case of $K$ algebraically closed can be resolved in a few lines
using geometric ideas such as ramification points and Hurwitz's Theorem.

Thus, let $R(x)$ be an arbitrary quadratic expression over an algebraically closed field $K$,
which we first assume not to have characteristic two.
Then according to Hurwitz's Theorem $R(x)$ has exactly two ramification points, each with index two.
By composing on both sides with suitable M\"obius transformations we may assume the
ramification points to be $\infty$ and $0$, and also that they coincide with the corresponding branch points.
In particular, $R(x)$ is equivalent to a quadratic polynomial (a quadratic
rational map without poles) and, actually, a scalar multiple of the map $x^2$ because
it has $0$ as a ramification point.
We conclude that over an algebraically closed field of characteristic different from two
every quadratic expression is equivalent to $x^2$.

Now assume the algebraically closed field $K$ has characteristic two.
Here it is possible for the quadratic expression $R(x)$ to be inseparable.
This occurs precisely when $R(x)$ is a square in $K(x)$,
and clearly all such expressions $R(x)$ are equivalent to $x^2$.
If $R(x)$ is separable then according to Hurwitz's Theorem it has precisely one ramification point.
Assuming both that and the corresponding branch point to be $\infty$, as we may,
$R(x)$ becomes a polynomial which is not a square, which can then be easily seen to be equivalent to $x^2+x$.
Thus, over an algebraically closed field of characteristic two
every quadratic expression is equivalent either to $x^2$ or to $x^2+x$.
Instead of $x^2+x$ one may take the equivalent expression
$
(1/x)\circ(x^2+x)\circ\bigl(x/(x+1)\bigr)=x+1/x,
$
which has $1$ as ramification point and $0$ as the corresponding branch point.

In the rest of this section we compute the lengths of the orbits of $G(\F_q)$ on the set of quadratic expressions over $\F_q$.
We start with the case of odd characteristic.

\begin{theorem}\label{thm:quadratic_finite_count}
Let the finite field $\F_q$ have odd characteristic, and fix a nonsquare $\sigma$ in $\F_q^{\ast}$.
Of all the $q^3(q^2-1)$ distinct quadratic rational expressions over $\F_q$, precisely
\begin{itemize}
\item[(i)]
$q^2(q^2-1)(q+1)/2$ are equivalent to $x^2$;
\item[(ii)]
$q^2(q^2-1)(q-1)/2$ are equivalent to $(x^2+\sigma)/(2x)$.
\end{itemize}
\end{theorem}

The expression $(x^2+\sigma)/(2x)$ of case (ii) of Theorem~\ref{thm:quadratic_finite_count} is clearly equivalent to $(x^2+\sigma)/x=x+\sigma/x$,
but the former matches a more general result which we will prove later, namely, Theorem~\ref{thm:x^r_finite}.

\begin{proof}
According to the above discussion of the algebraically closed field case,
any quadratic expression $R(x)$ over $\F_q$ is equivalent to $x^2$ over $\barFq$.
If the ramification points of $R(x)$ belong to $\F_q$, and hence the corresponding branch points as well,
then the same argument outlined above for the algebraically closed case shows that $R(x)$ is actually equivalent to $x^2$ over $\F_q$.

Let $\tau$ be a square root of $\sigma$ in $\F_{q^2}$, and consider
\[
\tau\frac{x+\sigma}{x-\sigma}
\circ x^2\circ
\tau\frac{x+\tau}{x-\tau}
=
\frac{x^2+\sigma}{2x}.
\]
This quadratic expression is equivalent to $x^2$ over $\barFq$, and actually already over $\F_{q^2}$,
but certainly not over $\F_q$ as its ramification points $\tau$ and $-\tau$ do not belong to $\F_q$, with $\tau$ and $-\tau$ as corresponding branch points.

Now if the ramification points of $R(x)$ do not belong to $\F_q$, then the $q$th power Galois automorphism interchanges them,
and so it does to the corresponding branch points.
Suitable M\"obius transformations over $\F_{q^2}$ take those ramification points to $\tau$ and $-\tau$,
and similarly for the branch points.
We refer the reader to the proof of Theorem~\ref{thm:x^r_finite}, which deals with a more general situation,
for an argument showing that those M\"obius transformations can be taken to be defined over $\F_q$,
implying that $R(x)$ is equivalent to $(x^2+\sigma)/(2x)$ as desired.
In the present proof we bypass that fact by relying on the orbit length count below.

\quad(i)\quad
Each element of the stabilizer of $x^2$, that is, a pair $(B,A)\in G(\F_q)$ such that $B(x)\circ x^2\circ A^{-1}(x)=x^2$,
must either fix or interchange the two ramification points of $x^2$, and similarly for the branch points.
This fact and a simple calculation show that the stabilizer consists of the pairs $(B,A)=(a^2x,ax)$
and $(a^2/x,a/x)$, with $a\in\F_q^\ast$.
Consequently, the stabilizer has order $2(q-1)$, whence the orbit of $x^2$ has length
$|G(\F_q)|/\bigl(2(q-1)\bigr)=q^2(q^2-1)(q+1)/2$.

\quad(ii)\quad
Consider now the quadratic expression $x\mapsto (x^2+\sigma)/(2x)$.
As in case (i) each element $(B,A)\in G(\F_q)$ of the stabilizer must either fix or interchange the two ramification points $\tau$ and $-\tau$,
and also the branch points.
The M\"obius transformations which fix both $\tau$ and $-\tau$ are
$(ax+b\sigma)/(bx+a)$,
and those who interchange them are
$(ax-b\sigma)/(bx-a)$,
for $a,b\in\F_q$, not both zero.
One then easily finds that the stabilizer consists of all pairs
\[
(B,A)=
\biggl(
\frac{(a^2+\sigma b^2)y+2\sigma a b}{2abx+(a^2+\sigma b^2)},
\frac{ax+\sigma b}{bx+a}
\biggr)
\]
and of all pairs
\[
(B,A)=
\biggl(
\frac{(a^2+\sigma b^2)y-2\sigma ab}{2aby-(a^2+\sigma b^2)},
\frac{ax-\sigma b}{bx-a}
\biggr)
\]
for $a,b \in\F_q$, not both zero.
Because proportional pairs $a,b$ yield the same stabilizer element,
the stabilizer has order $2(q^2-1)/(q-1) = 2(q+1)$.
Hence the orbit has length
$|G(\F_q)|/\bigl(2(q+1)\bigr)=q^2(q^2-1)(q-1)/2$.

Because the lengths of the two orbits which we have described add up to $q^3(q^2-1)$
we conclude that there are no further orbits.
\end{proof}

According to~\cite[Theorem~2]{MatPiz:self-reciprocal} and the discussion which follows it,
every quadratic expression over an arbitrary field $K$ of characteristic different from two is equivalent to $x+\sigma/x$,
with $\sigma\in K^{\ast}$ uniquely determined up to squares.

Now we consider the case of finite fields of characteristic two.

\begin{theorem}\label{thm:quadratic_finite_count_2}
Let $\F_q$ be a finite field of characteristic two.
Of all the $q^3(q^2-1)$ distinct quadratic rational expressions over $\F_q$, precisely
\begin{itemize}
\item[(i)]
$q(q^2-1)$ are equivalent to $x^2$;
\item[(ii)]
$q(q^2-1)^2$ are equivalent to $(x^2+1)/x$.
\end{itemize}
\end{theorem}

\begin{proof}
We know from the algebraically closed field case that
any quadratic expression $R(x)$ over $\F_q$ is equivalent to $x^2$ or $x+1/x$ over $\barFq$,
according to whether it is inseparable or separable.
One could now argue directly why
this equivalence actually holds over $\F_q$ in either case (as in the proof of Theorem~\ref{thm:x^r_finite}),
but we can dispense with that as the desired conclusion will follow from a calculation of the orbit lengths under $G(\F_q)$.

\quad(i)\quad
In this case the stabilizer of the quadratic expression $x^2$ is larger than in the odd characteristic case,
and is easily found to consist of all pairs
\[
(B,A)=
\biggl(
\frac{a^2x+b^2}{c^2x+d^2},
\frac{ax+b}{cx+d}
\biggr),
\]
with $a,b,c,d\in\F_q$ and $ad-bc\neq 0$.
Hence this stabilizer is isomorphic with $\PGL_2(\F_q)$, and so it has order $q(q^2-1)$,
and the orbit has length
$|G(\F_q)|/\bigl(q(q^2-1)\bigr)=q(q^2-1)$.

\quad(i)\quad
Instead of working with the quadratic expression $(x^2+1)/x$ it is convenient to consider the equivalent expression $x^2+x$.
For every element $(B,A)\in G(\F_q)$ in the stabilizer of the latter, $A$ must fix the only ramification point $\infty$,
and $B$ must fix the only branch point $\infty$.
Consequently, each of $A$ and $B$ is a degree one binomial, and then a simple calculation shows
that the stabilizer of $x^2+x$ consists of all pairs $(A,B)=(x+b+b^2,x+b)$.
with $b\in\F_q$.
Therefore, the stabilizer of $x^2+x$ is isomorphic with the additive group of $\F_q$.
Because the stabilizer has order $q$, the orbit has length $|G(\F_q)|/q=q(q^2-1)^2$.

As the lengths of the two orbits above add up to $q^3(q^2-1)$ we conclude that there are no further orbits.
\end{proof}

Note that, differently from the case of odd $q$, only a minority of the quadratic rational expressions
over a large finite field $\F_q$ of characteristic two are equivalent to $x^2$.
Thus, a `random' quadratic rational expressions over a large finite field $\F_q$ of characteristic two
is very likely to be equivalent to $x+1/x$, while that occurs with a probability approaching $1/2$ in the odd $q$ case.

According to~\cite[Theorem~2]{MatPiz:self-reciprocal} and the discussion which follows it,
every quadratic expression over an arbitrary field $K$ of characteristic two is equivalent either to $x^2$,
or to $x+\sigma/x$, with $\sigma\in K^{\ast}$ uniquely determined up to squares.

\section{Cubic expressions over an algebraically closed field}\label{sec:cubic}

In this section we find representatives for the equivalence classes, or $G(K)$-orbits,
of cubic rational expressions $R(x)$ over an algebraically closed field $K$.
Because a cubic expression has up to four ramification points, the description will involve
the configurations of quadruple of points of the projective line $\Proj^1(K)$, which are characterized by their cross-ratio.

Associated to an ordered quadruple of points of $\Proj^1(K)$ is their cross-ratio
\[
\lambda=(x_1,x_2;x_3,x_4):=\frac{(x_1-x_3)(x_2-x_4)}{(x_2-x_3)(x_1-x_4)}
\]
which is computed with respect to any affine coordinate $x$ on $\Proj^1(K)$.
The cross-ratio is an invariant of {\em ordered} quadruples of the projective line, in the sense that it
is left unchanged by all automorphisms of $\Proj^1(K)$, that is, by all M\"obius transformations.
The cross-ratio is also unaffected by simultaneously interchanging its entries in disjoint pairs.
More generally, when the four points are permuted, their cross-ratio changes within its orbit under the action of the
group $S$ of order six generated by the involutions $\lambda\mapsto 1/\lambda$ and $\lambda\mapsto 1-\lambda$.
The remaining elements of $S$, besides the identity, are the involution $\lambda\mapsto\lambda/(\lambda-1)$,
the substitution
$\lambda\mapsto 1/(1-\lambda)$ of order three, and its inverse $\lambda\mapsto (\lambda-1)/\lambda$.
Thus, as an invariant of {\em unordered} quadruples of the projective line, the cross-ratio takes its values
in the set of $S$-orbits on $K\cup\{\infty\}$.

Note that the orbits of $S$ on $K\cup\{\infty\}$ have length six,
with the only exceptions of the orbits $\{\infty,0,1\}$ and $\{1/2,2,-1\}$ of length three (which coincide if
$K$ has characteristic two, and the latter orbit has just one element if $K$ has characteristic three), and possibly
one further orbit consisting of the two roots of $\lambda^2-\lambda+1$ in $K$.

For $K$ of characteristic different from two and three we will also encounter the involution $\lambda\mapsto(\lambda-2)/(2\lambda-1)$.
Because that commutes with $\lambda\mapsto 1/\lambda$
and with $\lambda\mapsto 1-\lambda$, it commutes with each element of $S$.
In fact, together with the identity it forms the centralizer of $S$ in the M\"obius group.
Consequently, the involution $\lambda\mapsto(\lambda-2)/(2\lambda-1)$ permutes the $S$-orbits on $K\cup\{\infty\}$.
One can verify that the only $S$-orbits fixed by it are the set of roots of $\lambda^2-\lambda+1$ (if contained in $K$),
and the set of roots of
$(\lambda^2-4\lambda+1)(\lambda^2+2\lambda-2)(2\lambda^2-2\lambda-1)$
(if contained in $K$).

We are now ready to state and prove a classification of cubic rational maps up to equivalence over an algebraically closed field,
postponing the cases of characteristic two and three to later consideration.

\begin{theorem}\label{thm:cubic_expr_alg_closed}
Let $K$ be an algebraically closed field of characteristic different from two and three.
Any cubic rational expression $R(x)$ over $K$ is equivalent either to $x^3$, or to
\[
R_c(x):=
\frac{x^3+(c-2)x^2}{(2c-1)x-c},
\]
for some $c\in K\setminus\{0,1\}$.

The cubic expressions $R_c(x)$ and $R_{c'}(x)$, for $c,c'\in K\setminus\{0,1\}$, are equivalent if and only if $c$ and $c'$ belong to the same $S$-orbit.

The cubic expression $R_c(x)$ has ramification points
$\infty$, $0$, $1$, $\lambda$, with corresponding branch
points $\infty$, $0$, $1$, $\mu$, where
$\lambda=c(2-c)/(2c-1)$ and
$\mu=c(2-c)^3/(2c-1)^3$.
The four ramification points are distinct,
and so are the four branch points,
except when $c\in\{1/2,2,-1\}$.
Those three exceptional cases give expressions equivalent to
$R_{1/2}(x)=-2x^3+3x^2$,
whose ramification points are $\infty,0,1$, with $\infty$ of index three.
\end{theorem}

\begin{proof}
According to Hurwitz's formula quoted in Theorem~\ref{thm:Hurwitz}, the cubic expression
$R(x)$ has at most four ramification points $P_i$, with ramification indices
$(3,3)$, $(3,2,2)$, or $(2,2,2,2)$,
and corresponding distinct branch points $R(P_i)$.

If $R(x)$ has only two ramification points, say $P_1$ and $P_2$, then
after pre- and post-composition with suitable automorphisms of $\Proj^1(K)$
we may assume those to be $0$ and $\infty$, and the corresponding branch points to be $0$ and $\infty$.
More precisely, if $A(x)$ is any M\"obius transformation over $\barFq$
which maps $0$ to $P_1$, and $\infty$ to $P_2$,
then $(R\circ A)(x)$ has ramification points $0$ and $\infty$.
Similarly, if $B(x)$ is any M\"obius transformation over $\barFq$
which maps the corresponding branch points $R(P_1)$ and $R(P_2)$
to $0$ and $\infty$, respectively, then $(B\circ R\circ A)(x)$ has ramification points $0$ and $\infty$, each with index three,
and with corresponding branch points $0$ and $\infty$.
Clearly $(B\circ R\circ A)(x)$ is then a scalar multiple of $x^3$, and so it is equivalent to $x^3$.

Now suppose that $R(x)$ has at least three  ramification points.
Assuming, as we may, that the ramification points include $\infty$ and $0$, with corresponding branch points $\infty$ and $0$
(hence it has $\infty$ as at least a double pole and $0$ as at least a double zero),
$R(x)$ will have the form
$R(x)=(x^3+ax^2)/(bx-c)$,
for some $a,b,c\in K$.
Further imposing, as we may, that $R(x)$ has $1$ as a ramification point with
corresponding branch point $1$, amounts to $R(x-1)-1$
having a double zero at $0$.
A short calculation then leads to
\[
R(x)=R_c(x):=
\frac{x^3+(c-2)x^2}{(2c-1)x-c},
\]
with $c\neq 0,1$ (otherwise $R_c(x)$ would be quadratic rather than cubic).
Because the ramification points besides $\infty$ are the zeroes of the derivative
$R'(x)$, one computes that the fourth ramification point (which may coincide with one of $\infty,0,1$) is
$\lambda=c(2-c)/(2c-1)$, with corresponding branch point
$\mu=c(2-c)^3/(2c-1)^3$.
A similar discussion of an equivalent family is given in~\cite[Example~8.2]{Osserman:given_ramification}.

Now we examine when $R(x)$ has precisely three ramification points.
We see that $\lambda=\infty,0,1$ when
$c=1/2,2,-1$, respectively.
Because $\mu$ equals $\lambda$ for each of those values of $c$, those three choices of $c$ yield equivalent rational expressions $R_c(X)$,
as one can change from one another by suitably permuting the three ramification points,
and the branch points correspondingly.

It remains to determine when $R_c(x)$ and $R_{c'}(x)$ are equivalent
for $c,c'\in K\setminus\{0,1,1/2,2,-1\}$.
A necessary condition is that the corresponding cross-ratios $\lambda$ and $\lambda'$ belong to the same $S$-orbit,
but we will see that this alone is not sufficient.
Because the substitutions $c\mapsto 1/c$ and $c\mapsto 1-c$ induce
$\lambda\mapsto 1/\lambda$ and $\lambda\mapsto 1-\lambda$, respectively,
any substitution on $c$ in the group $S$ induces the analogous substitution on the corresponding cross-ratio $\lambda$.
In other words, the quadratic map sending $c$ to $\lambda=c(2-c)/(2c-1)$ is $S$-equivariant.
In particular, it maps $S$-orbits of $c$ to $S$-orbits of $\lambda$.
Each value of $\lambda\in K\setminus\{0,1\}$ originates from exactly two choices of $c\in K\setminus\{0,1,1/2,2,-1\}$, except when $c^2-c+1=0$,
which is equivalent to $\lambda^2-\lambda+1=0$.
The involution $c\mapsto (c-2)/(2c-1)$ that we introduced before Theorem~\ref{thm:cubic_expr_alg_closed}
interchanges those two choices of $c$ for a given $\lambda$.
However, those two choices give inequivalent expressions $R_c(x)$ because they yield different values of
$\mu=\lambda^3/c^2$.
(Here we are not allowed to permute the branch points for computing their cross-ratio $\mu$,
having already chosen an ordering for the respective ramification points to compute $\lambda$.)
We conclude that $R_c(x)$ and $R_{c'}(x)$ are equivalent precisely when $c$ and $c'$ belong to the same $S$-orbit.
\end{proof}

\begin{rem}\label{rem:cubic_expr_alg_closed}
A variant of the parametric family $R_c(x)$ of Theorem~\ref{thm:cubic_expr_alg_closed} is
\[
(-4x+2)\circ R_c(x)\circ\frac{x+1}{2}=
\frac{x^3-bx^2-3x-b}{bx+1},
\]
where $b=1-2c$.
This has ramification points $\infty$, $-1$, $1$, $(b^2-3)/2b$,
with corresponding branch points $\infty$, $2$, $-2$, $-(b^4+18b^2-27)/4b^3$.
For $c=1/2$ this expression reads $x^3-3x$, which will be a more convenient choice than $R_{1/2}(x)=-2x^3+3x^2$
when passing to a finite ground field in Theorem~\ref{thm:cubic_expr_finite}.
Note that $x^3-3x=D_3(x,1)$ is a Dickson polynomial, where
$D_k(x+\alpha/x,\alpha)=x^k+\alpha/x^k$.
\end{rem}

We now consider the cases of characteristics three and two.
Besides the possibility that $R(x)$ is inseparable, we must take into account that
Hurwitz's Theorem gives a strict inequality, rather than an equality, in case some ramification index is a multiple of the characteristic.

\begin{theorem}\label{thm:cubic_expr_alg_closed_3}
Let $K$ be an algebraically closed field of characteristic three.
Any cubic rational expression $R(x)$ over $K$ is equivalent to either $x^3$, or $x^3+x^2$, or $x^3+x$, or
\[
R_c(x):=
\frac{x^3+(c+1)x^2}{-(c+1)x-c},
\]
for some $c\in K\setminus\F_3$.
\end{theorem}

The parametric expression for $R_c(x)$ equals that of Theorem~\ref{thm:cubic_expr_alg_closed} viewed modulo three.
Also, for the excluded value $c=-1$, the formula for $R_c(x)$ returns the inseparable cubic expression $R_{-1}(x)=x^3$ (which is ramified everywhere).
The remaining rational expressions given have, respectively, $(3,2)$, $(3)$, and $(2,2,2,2)$ for the parametric family,
as lists of ramification indices.
These exhaust all the possibilities allowed by Hurwitz's Theorem in characteristic three.

\begin{proof}
We resort to direct calculation to cover cases where Hurwitz's formula does not provide an equality.
Thus, suppose that $R(x)$ has a ramification point with index $3$, which we can assume to be $\infty$, with image $\infty$.
Hence $R(x)$ is a cubic polynomial, which is then equivalent to a monic one without constant term.
In turn, $R(x)$ is easily seen to be equivalent to either $x^3$ (which is inseparable), or $x^3+x^2+bx$, or $x^3+bx$, for some $b\in K$.
In the second case $R(x)$ is equivalent to
\[
(x-b^3-2b^2)\circ(x^3+x^2+bx)\circ(x-b/2)
=x^3+x^2,
\]
which has $0$ as a further ramification point, with index $2$.
In the third case $R(x)$ is equivalent to
\[
(x/\sqrt{b}\,^3)\circ(x^3+bx)\circ(\sqrt{b}\,x)
=x^3+x
\]
(with $\sqrt{b}$ denoting a fixed square roots of $b$ in $K$),
whose only ramification point is $\infty$.

Now we may assume that all ramification points of $R(x)$ have indices less than $3$, whence equality holds in Hurwitz's formula.
Consequently, $R(x)$ has four ramification points, each with index $2$.
The corresponding part of the proof of Theorem~\ref{thm:cubic_expr_alg_closed}
applies and yields the desired conclusion.
\end{proof}

\begin{rem}\label{rem:cubic_expr_alg_closed_3}
As in Remark~\ref{rem:cubic_expr_alg_closed}, a variant of the parametric family $R_c(x)$ of Theorem~\ref{thm:cubic_expr_alg_closed_3} is
$(x^3-bx^2-b)/(bx+1)$,
where $b=c+1\in K\setminus\F_3$.
This has ramification points $\infty$, $-1$, $1$, $-b$,
with corresponding branch points $\infty$, $-1$, $1$, $-b$.
Hence with this choice the ramification points coincide with the four fixed points of the cubic expression.
\end{rem}

We conclude this section by considering the case where the field $K$ has characteristic two.
Then Theorem~\ref{thm:Hurwitz} implies that $R(x)=g(x)/h(x)$ can only have at most two ramification points.
An elementary way of seeing this is noting that,
because the second derivative of every rational expression in characteristic two vanishes,
the derivative
$R'(x)=\bigl(g'(x)h(x)-g(x)h'(x)\bigr)/h(x)^2$
must be a rational expression in $x^2$,
and hence its numerator, which is a polynomial of degree at most four,
can only have at most two distinct zeroes.

\begin{theorem}\label{thm:cubic_expr_alg_closed_2}
Let $K$ be an algebraically closed field of characteristic two.
Any cubic rational expression $R(x)$ over $K$ is equivalent to either
$x^3$, or
$x^3+x^2$, or
$(x^3+1)/x$, or
$(x^3+cx^2)/(x+1)$ for a unique $c\in K\setminus{\F_2}$.
\end{theorem}

The lists of ramification indices for the above expressions are, respectively,
$(3,3)$, $(3,2)$, $(2)$, and $(2,2)$ for the parametric family.

\begin{proof}
The expression $R(x)$ must have at least one ramification point, so we may assume that $\infty$ is a ramification point,
and that $\infty$ is also the corresponding branch point.

If $\infty$ has ramification index $3$ then $R(x)$ is a polynomial, which we may assume monic, say $R(x)=x^3+ax^2+bx+c$.
Because the remaining ramification points are the roots of its derivative $x^2+b$,
there is precisely one further ramification point, namely $\sqrt{b}$.
We may now assume that this ramification point is $0$, and also that $0$ is the corresponding branch point.
Then $R(x)=x^3+ax^2$, and when $a\neq 0$ this is equivalent to $x^3+x^2$.

Now we may suppose that $\infty$ has ramification index $2$.
We may also assume that the other preimage of $\infty$ is $0$.
Further multiplying $R(x)$ by a scalar we find that it is equivalent to $(x^3+ax^2+bx+c)/x$, for some $a,b,c\in K$ with $c\neq 0$.

If $R(x)$ has no further ramification points besides $\infty$, then the expression
\[
\frac{d}{dx}\frac{x^3+ax^2+bx+c}{x}=\frac{ax^2+c}{x^2}
\]
has no zeroes in $K$, forcing $a=0$ (and of course $c\neq 0$).
Now $(x^3+bx+c)/x$ is clearly equivalent to $(x^3+c)/x$ by adding a constant, and finally to $(x^3+1)/x$ as desired.

Finally, suppose $R(x)$ has a further ramification point besides $\infty$.
We may assume that to be $1$, whence $a=c\neq 0$, and the corresponding branch point to be $0$, whence $b=1$.
Hence $R(x)$ is equivalent to $(x^2+1)(x+a)/x$.
When $a=1$ we get $(x+1)^3/x$, which is ramified at $1$ with index $3$, and hence is equivalent to $x^3+x^2$.
Otherwise we get a parametric family with $a\in K\setminus{\F_2}$, having ramification indices $(2,2)$.

Different values of the parameter $a$ yield inequivalent expressions.
In fact, the parameter $a=(\infty,0;1,a)$ is determined as the cross-ratio of the quadruple consisting
of the preimages of the two branch points, of which two are the ramification points.
Note that such cross-ratio is independent of which of the two ramification points was initially chosen to be $\infty$,
as interchanging the first two entries and simultaneously the last two entries of a cross-ratio leaves it unchanged.

The variant of this parametric family given in Theorem~\ref{thm:cubic_expr_alg_closed_2} is obtained as
$(x^2+1)(x+a)/x\circ(x+1)=(x^3+cx^2)/(x+1)$, where $c=a+1\in K\setminus{\F_2}$.
\end{proof}

\begin{rem}\label{rem:cubic_expr_alg_closed_2}
The parametric family of
Theorem~\ref{thm:cubic_expr_alg_closed_2}
bears some resemblance to the family $R_c(x)$ of Theorem~\ref{thm:cubic_expr_alg_closed}.
However, the latter would simplify to the quadratic expression $(x^3+cx^2)/(x+c)=x^2$ when viewed modulo $2$.
A variant of the family of
Theorem~\ref{thm:cubic_expr_alg_closed_2} is
$(x^3+x^2)/(x+c')$, with $c'\in K\setminus\F_2$.
One further variant for the parametric family, which is better suited
to the purpose of passing to a finite ground field in Section~\ref{sec:small_char}, is
\[
(x/c^2+1)
\circ
\frac{x^3+cx^2}{x+1}
\circ
(cx+c)
=\frac{x^3+c''}{x+c''},
\]
where $c''=1+1/c$,
hence again $c''\in K\setminus\F_2$, uniquely determined by the equivalence class.
These cubic expressions have $\infty$ and $1$ as ramification points of index $2$, and also as the corresponding branch points.
The further preimages of the branch points are $c''$ and $0$ in this parametrization.
\end{rem}

\section{Cubic expressions over a finite field}\label{sec:cubic_finite}
In this section we find representatives for the
equivalence classes of cubic rational expressions $R(x)$ over a finite field,
restricting to those of expressions with at most three distinct ramification points.
The idea is to view $R(x)$ over the algebraic closure $\barFq$ of $\F_q$ and then apply the corresponding classification result from the previous section.
Here we view $\Proj^1(\F_q)$ as a subset of $\Proj^1(\barFq)$, but in our context we may also identify them
with $\F_q\cup\{\infty\}$ and $\barFq\cup\{\infty\}$, respectively.
The Galois group of $\barFq/\F_q$, which is topologically generated by $\alpha\mapsto\alpha^q$, acts on the latter, with the former as the set of fixed points.
The Galois group permutes the ramification points of $R(x)$ over $\barFq$, and the corresponding branch points in a matching way, preserving the ramification indices.

We start with cubic rational expressions having only two ramification points in $\barFq\cup\{\infty\}$,
hence equivalent to $x^3$ over $\barFq$ according to Theorem~\ref{thm:cubic_expr_alg_closed}.
In this case it is not any harder to deal with expressions of arbitrary degree $r$ having precisely two ramification points in $\barFq$.

\begin{theorem}\label{thm:x^r_finite}
Let $r$ be a positive integer, let $\F_q$ be a finite field of odd characteristic not dividing $r$,
and let $\sigma\in\F_q$ be a nonsquare.
Let $R(x)$ be a rational expression over $\F_q$, of degree $r$ and with only two ramification points in $\barFq$.
Then $R(x)$ is equivalent over $\F_q$ to either $x^r$, or
\[
\biggl(\sum_{h}\binom{r}{2h}x^{r-2h}\sigma^h\biggr)
\bigg/
\biggl(\sum_{h}\binom{r}{2h+1}x^{r-2h-1}\sigma^h\biggr).
\]
\end{theorem}

The two special forms for $R(x)$ given in Theorem~\ref{thm:x^r_finite} are inequivalent, as the ramification points of the latter are not all in  $\Proj^1(\F_q)$.
Note that the latter form reads $(x^2+\sigma)/(2x)$ when $r=2$, as in Theorem~\ref{thm:quadratic_finite_count}, and
$(x^3+3\sigma x)/(3x^2+\sigma)$ when $r=3$, which will appear in Theorem~\ref{thm:cubic_expr_finite}.

\begin{proof}
View $R(x)$ over the algebraic closure $\barFq$ of $\F_q$, and let $\tau$ be a square root of $\sigma$ in $\F_{q^2}$.
Because the ramification indices satisfy $e_R(P)\le\deg(R)=r$, according to Hurwitz's formula each of the two ramification points has index $r$.

If the two ramification points of $R(x)$ belong to $\Proj^1(\F_q)$ then $R(x)$ is equivalent to $x^r$ over $\F_q$.
This can be shown exactly as for the corresponding case in the proof of Theorem~\ref{thm:cubic_expr_alg_closed} (for $r=3$ and over $\barFq$),
just noting that the M\"obius transformations $A(x)$ and $B(x)$ used there can be taken to be defined over $\F_q$.

Before dealing with the remaining case note that
\begin{equation*}
\tau\frac{x+\tau^r}{x-\tau^r}
\circ x^r\circ
\tau\frac{x+\tau}{x-\tau}
=
\frac{
\sum_{\text{$k$ even}}\binom{r}{k}x^{r-k}\tau^k
}{
\sum_{\text{$k$ odd}}\binom{r}{k}x^{r-k}\tau^{k-1}
}
=
\frac{
\sum_{h}\binom{r}{2h}x^{r-2h}\sigma^h
}{
\sum_{h}\binom{r}{2h+1}x^{r-2h-1}\sigma^h
}
\end{equation*}
is equivalent to $x^r$ over $\barFq$, but has $\tau$ and $-\tau$ as ramification points, with $\tau$ and $-\tau$ as corresponding branch points.

Now suppose that two ramification points of $R(x)$ do not belong to $\Proj^1(\F_q)$.
Then they must be Galois-conjugated over $\F_q$, say $\rho,\rho^q\in\barFq\setminus\F_q$.
There is a unique M\"obius transformation $A(x)$ over $\barFq$
which maps $\tau$ to $\rho$, $-\tau$ to $\rho^q$, and fixes some other point of $\Proj^1(\F_q)$, say $0$.
But $A(x)$ is actually defined over $\F_q$, because the associated map $\Proj^1(\barFq)\to\Proj^1(\barFq)$
commutes with the Galois automorphism $\alpha\mapsto\alpha^q$ of $\barFq$.
Now $(R\circ A)(x)$ has ramification points $\tau$ and $-\tau$.
Similarly, there is a unique M\"obius transformation $B(x)$ over $\barFq$
which maps the branch points $R(\rho)$ and $R(-\rho)$ of $R(x)$
(or of $(R\circ A)(x)$, for that matter) to $\tau$ and $-\tau$, respectively,
and maps $R(0)$ to $0$, and similarly $B(x)$ must also be defined over $\F_q$.
Now $(B\circ R\circ A)(x)$ has ramification points $\tau$ and $-\tau$, both with index $r$,
and with corresponding branch points $\tau$ and $-\tau$, and maps $0$ to $0$.
But then $(B\circ R\circ A)(x)$ equals the expression given in Theorem~\ref{thm:x^r_finite},
because such conditions uniquely characterize the latter among rational expressions of degree $r$ over $\barFq$,
a fact which can be seen at once on its equivalent expression $x^r$.
\end{proof}

\begin{rem}\label{rem:cubic-invariant}
According to~\cite[Lemma~11]{MatPiz:self-reciprocal},
a polynomial $F(x)$ of degree $3n$ over a field satisfies $(x-1)^{3n}\cdot F\bigl(1/(1-x)\bigr)=F(x)$
precisely if it has the form
$
F(x)=(x^2-x)^n\cdot f\bigl((x^3-3x+1)/(x^2-x)\bigr)
$
for some polynomial $f(x)$ of degree $n$.
This can be thought of a cubic version of the interpretation of self-reciprocal polynomials as arising through a quadratic transformation
recalled in the Introduction.
However, such polynomial $F(x)$ is very far from the general case of a polynomial arising through an arbitrary cubic transformation,
as $(x^3-3x+1)/(x^2-x)$ has precisely two ramification points
(in characteristic not three, each with ramification index three), namely, the roots of $x^2-x+1$.
In particular, according to Theorem~\ref{thm:x^r_finite}, over a field $\F_q$ with $q-1$ a multiple of three,
that cubic expression is equivalent to $x^3$.
\end{rem}

Now we are ready to classify cubic expressions over $\F_q$ having at most three ramification points.

\begin{theorem}\label{thm:cubic_expr_finite}
Let $\F_q$ be a finite field of characteristic at least five,
and let $\sigma\in\F_q$ be a nonsquare.
Then any cubic rational expression $R(x)$ over $\F_q$ with at most three ramification points (over $\barFq$) is equivalent to either
$x^3$, or $(x^3+3\sigma x)/(3x^2+\sigma)$, or $x^3-3x$, or $x^3-3\sigma x$.
\end{theorem}

\begin{proof}
View $R(x)$ over the algebraic closure $\barFq$ of $\F_q$, and let $\tau$ be a square root of $\sigma$ in $\F_{q^2}$.
If $R(x)$ has only two ramification points in $\Proj^1(\barFq)$, then Theorem~\ref{thm:x^r_finite} applies.
Hence $R(x)$ is equivalent to either $x^3$ or to $(x^3+3\sigma x)/(3x^2+\sigma)$ over $\F_q$,
depending on whether its ramification points belong to $\Proj^1(\F_q)$ or not.

Now suppose that $R(x)$ has three ramification points in $\Proj^1(\barFq)$.
Then according to Theorem~\ref{thm:cubic_expr_alg_closed} and Remark~\ref{rem:cubic_expr_alg_closed} it is equivalent to $x^3-3x$ over $\barFq$.
The latter has ramification points $\infty$ with index $3$, and $\pm 1$ with index $2$, and corresponding branch points $\infty$ and $\mp 2$.
The polynomial $x^3-3\sigma x$ has $\infty$  as a ramification point with index $3$,
and $\tau$ and $-\tau$ as ramification points with index $2$.
The corresponding branch points are $\infty$, $-2\tau^3$ and $2\tau^3$, respectively.
Hence $x^3-3\sigma x$ is also equivalent to $x^3-3x$ over $\barFq$ (or even over $\F_{q^2}$), but not over $\F_q$.
We now show that $R(x)$ is equivalent to either $x^3-3x$ or $x^3-3\sigma x$ over $\F_q$.

The unique ramification point of $R(x)$ with index $3$ must be fixed by the Galois automorphism $\alpha\mapsto\alpha^q$,
hence it belongs to $\Proj^1(\F_q)$, and so does the corresponding branch point.
If the other two ramification points belong to $\Proj^1(\F_q)$, and so obviously do the corresponding branch points,
then $R(x)$ is actually equivalent to $x^3-3x$ over $\F_q$.
This can be seen by composing $R(x)$ on both sides with appropriate M\"obius transformations so as to turn the
ramification points of $R(x)$, and also the corresponding branch points, into $\infty$, $0$, and $1$,
as in the proof of Theorem~\ref{thm:cubic_expr_alg_closed}.
In fact, the M\"obius transformations $A(x)$ and $B(x)$ employed to achieve that are actually defined over $\F_q$ because each takes
some particular triple of distinct points of $\Proj^1(\F_q)$ to another such triple.

Now suppose that $R(x)$ has three ramification points in $\Proj^1(\barFq)$, but those with index two are not in $\Proj^1(\F_q)$.
Then they will be Galois-conjugated over $\F_q$, say $\rho$ and $\rho^q$.
The argument is very similar to the one we used in the proof of Theorem~\ref{thm:x^r_finite}.
Thus, there is a unique M\"obius transformation $A(x)$ over $\barFq$
which maps $\tau$ to $\rho$, $-\tau$ to $\rho^q$, and the third ramification point of $R(x)$ to $0$.
There is also a unique M\"obius transformation $B(x)$ over $\barFq$
which maps the branch points $R(\rho)$ and $R(-\rho)$ of $R(x)$
to $-2\tau^3$ and $2\tau^3$, respectively, and the third branch point of $R(x)$ to $\infty$.
Both $A(x)$ and $B(x)$ are actually defined over $\F_q$.
Then $(B\circ R\circ A)(x)$ equals $x^3-3\sigma x$,
as the latter is uniquely characterized among cubic rational expressions
by having ramification points $\infty$, $\tau$, $-\tau$ with corresponding images $\infty$, $-2\tau^3$, $2\tau^3$,
and $\infty$ having ramification index three.
\end{proof}

 %

\begin{theorem}\label{thm:cubic_expr_finite_3}
Let $\F_q$ be a finite field of characteristic three,
and let $\sigma\in\F_q$ be a nonsquare.
Then any cubic rational expression $R(x)$ over $\F_q$ with at most three ramification points is equivalent to
either $x^3+x^2$, or $x^3+x$, or $x^3+\sigma x$.
\end{theorem}

\begin{proof}
According to Theorem~\ref{thm:cubic_expr_alg_closed_3}, the cubic expression $R(x)$ is equivalent over $\barFq$ to either $x^3+x^2$ or $x^3+x$,
according to whether it has ramification indices $(3,2)$ or $(3)$.
In both cases $R(x)$ has a unique ramification point of index $3$, which must therefore belong to $\Proj^1(\F_q)$.
Therefore $R(x)$ can be replaced with an equivalent one which has $\infty$ as ramification point with index $3$,
and corresponding branch point $\infty$.
We may now follow the corresponding part of the proof of Theorem~\ref{thm:cubic_expr_alg_closed_3},
where most reductions can be done over $\F_q$,
and we point out those which require modification.
Thus, $R(x)$ is a cubic polynomial, and being separable by hypothesis one finds that it is equivalent over $\F_q$
to either $x^3+x^2+bx$, or $x^3+bx$, for some $b\in\F_q$.
The former polynomial is seen to be equivalent to $x^3+x^2$
as in the proof of Theorem~\ref{thm:cubic_expr_alg_closed_3}.
The polynomial $x^3+bx$, however, is equivalent to
$(x/\sqrt{b}\,^3)\circ(x^3+bx)\circ(\sqrt{b}\,x)
=x^3+x$
only when $b$ is a square in $\F_q$.
In the complementary case one can similarly show that $R(x)$ is equivalent to $x^3+\sigma x$.
\end{proof}

Over a finite field $\F_q$ of characteristic three we also have inseparable cubic expressions,
which are ramified everywhere.
Those are easily seen to be all equivalent to $x^3$ already over $\F_q$.

\section{Bounding the number of equivalence classes}\label{sec:bound}

According to Lemma~\ref{lemma:coprime_probability} there are a total of $q^5(q^2-1)$ cubic rational expressions over $\F_q$.
They divide into a number of equivalence classes, that is,
orbits under the action of $G(\F_q)=\PGL_2(\F_q) \times \PGL_2(\F_q)$.
In Section~\ref{sec:cubic_finite} we have determined representatives for those orbits consisting of expressions with at most three ramification points.
In this section we will produce an upper bound for the total number of equivalence classes,
making use of the orbit-stabilizer theorem to estimate the size of classes where explicit representatives are not available.

The gist of our argument is that a lower bound on the length of an orbit of cubic rational expressions with four distinct ramification points
follows from an upper bound on the order of the stabilizer in $G(\F_q)$
of any expression $R(x)$ in the orbit.
The stabilizer of $R(x)$ permutes the four ramification points of $R(x)$ in $\barFq\cup\{\infty\}$,
and the four branch points correspondingly.
Thus, we may identify the stabilizer of $R(x)$ with a subgroup of the symmetric group $S_4$ on the four ramification points of $R(x)$.
Consequently, the stabilizer has order at most $|S_4|=24$, and hence the orbit has length at least $q^2(q^2-1)^2/24$.
However, it turns out that the stabilizer cannot have order larger than $4$, except for at most one orbit with stabilizer of order $12$.
To prove that we need the following auxiliary result.

\begin{lemma}\label{lemma:stabilizer}
Let $R(x)$ be a cubic rational expression over an algebraically closed field $K$,
and suppose $R(x)$ has four distinct ramification points.
Then the stabilizer of $R(x)$ in the action of $\PGL_2(K) \times \PGL_2(K)$
acts as the Klein Four-Group on the ramification points,
except if the cross-ratio of the ramification points is a primitive sixth root of unity.
If the latter occurs, then $R(x)$ is equivalent to $3x^2/(2x^3+1)$,
whose stabilizer acts as the alternating group $A_4$ on the ramification points.
\end{lemma}


\begin{proof}
Because $R(x)$ has four distinct ramification points, $K$ does not have characteristic two.
According to Theorems~\ref{thm:cubic_expr_alg_closed} and~\ref{thm:cubic_expr_alg_closed_3},
the cubic expression $R(x)$ is equivalent to
\[
R_c(x)
=
\frac{x^3+(c-2)x^2}{(2c-1)x-c}
\]
for some $c\in K\setminus\{0,1,1/2,2,-1\}$.
Now $R_c(x)$ has ramification points
$\infty$, $0$, $1$, $\lambda$, with corresponding branch
points $\infty$, $0$, $1$, $\mu$, where
$\lambda=c(2-c)/(2c-1)$ and
$\mu=c(2-c)^3/(2c-1)^3$.
Let $H$ be the stabilizer of $R_c(x)$ in the action of $G(K)$.
Every element of $H$ induces a permutation of the four ramification points in the domain of $R_c(x)$, and the same permutation of the four branch points in the codomain.

One may check that
\[
(\mu/x)\circ R_c(x)\circ(\lambda/x)=R_c(x)
\]
and
\[
\frac{x-\mu}{x-1}\circ R_c(x)\circ\frac{x-\lambda}{x-1}=R_c(x).
\]
Because $\lambda/x$ induces the permutation $(\infty,0)(1,\lambda)$ on the ramification points,
and $(x-\lambda)/(x-1)$ induces the permutation $(\infty,1)(0,\lambda)$,
we conclude that $H$ always contains the Klein Four-Group.

No element of $H$ can induce the transposition $(\infty,0)$ on the ramification points, because
\[
(1/x)\circ R_c(x)\circ(1/x)=R_{1/c}(x)
\]
equals $R_c(x)$ if and only if $c=\pm 1$, which are excluded values.
Consequently, $H$ cannot act as the full symmetric group $S_4$ on the four ramification points.

It also follows that $H$ cannot act as a dihedral group of order eight.
In fact, precisely one of the transpositions $(\infty,0)$, $(\infty,1)$, $(0,1)$
belongs to each of the three subgroups of order eight of the symmetric group on the four ramification points
$\infty$, $0$, $1$, $\lambda$ of $R_c(x)$.
We have already ruled out the first case.
If some element of the stabilizer $H$ of $R_c(x)$ induced
the transposition $(\infty,1)$ on the ramification points,
then the corresponding (conjugate) element of the stabilizer of
\[
\frac{1}{1-x}\circ R_c(x)\circ\frac{x-1}{x}
=R_{1/(1-c)}(x)
\]
would induce the transposition $(\infty,0)$ on the ramification points
$\infty$, $0$, $1$, $1/(1-\lambda)$
of $R_{1/(1-c)}(x)$,
but we have shown that to be impossible.
The third case is similarly ruled out by considering $R_{(c-1)/c}(x)$.

Thus, either $H$ acts as the Klein Four-Group on the four ramification points of $R_c(x)$, or as the alternating group $A_4$.
The latter occurs if and only if some element of $H$ acts as a $3$-cycle fixing the fourth ramification point $\lambda$.
That is the case precisely when
$R_{1/(1-c)}(x)$ equals $R_c(x)$,
hence when $c^2-c+1=0$.
Because $c\neq -1$, the characteristic of $K$ is different from three.
The condition $c^2-c+1=0$ is then equivalent to $\lambda^2-\lambda+1=0$,
which means that $\lambda$ is a primitive sixth root of unity.

To prove the final claim note that the four ramification points of $R(x)=3x^2/(2x^3+1)$ coincide with the corresponding branch points,
and they are $0$, $1$,
and the two primitive sixth roots of unity in $K$.
Because their cross-ratio is also a primitive sixth root of unity, $R(x)$ is equivalent to $R_c(x)$ with $c$ a primitive sixth root of unity.
\end{proof}

\begin{theorem}\label{thm:cubic_expr_bound}
Over a finite field $\F_q$ of odd characteristic,
there are no more than $4q$ equivalence classes of cubic rational expressions.
\end{theorem}

\begin{proof}
Suppose first that $q$ is not a power of three.
According to Theorem~\ref{thm:cubic_expr_finite}, under the action of $G(\F_q)$
there are precisely four orbits of cubic rational expressions over $\F_q$ having at most three ramification points.
We will use the orbit-stabilizer theorem to compute the length of those orbits, and then by difference to bound the number
of orbits of cubic rational expressions with four ramification points.

Similarly to the proof of Theorem~\ref{thm:quadratic_finite_count} one finds that the stabilizer of the cubic expression $x^3$ has order $2(q-1)$.
A calculation shows that the stabilizer of the other cubic expression with only two ramification points,
namely $(x^3+3\sigma x)/(3x^2+\sigma)$, where $\sigma\in\F_q$ is a nonsquare, has order $2(q+1)$.
Consider now the two expressions $x^3-3x$ and $x^3-3\sigma x$.
Each has two (opposite) ramification points with index two, and one with index three.
Each element of a stabilizer must fix this third ramification point and permute the other two (and the branch points accordingly),
hence it is either the identity or the map given by pre- and post-composition with $-x$.
Hence in both cases the stabilizer has order $2$.

Thus far, we have seen two orbits consisting of $q^2(q^2-1)(q+1)/2$ and $q^2(q^2-1)(q-1)/2$ cubic expressions,
and another two orbits consisting of $q^2(q^2-1)^2/2$ expressions each.
Adding up these numbers, we find $q^2(q^2-1)(q^2+q-1)$ cubic rational expressions having two or three ramification points.
Subtracting this from the total number $q^5(q^2-1)$ of cubic expressions
we find that, over $F_q$ of characteristic at least five, the cubic expressions with four ramification points are in number of
$q^2(q^2-1)^2(q-1)$.

Now suppose $\F_q$ has characteristic three, so Theorem~\ref{thm:cubic_expr_finite_3} applies.
As in case (i) of Theorem~\ref{thm:quadratic_finite_count_2}, the stabilizer of $x^3$ in $G(\F_q)$ is isomorphic to $\PGL_2(\F_q)$,
and hence the orbit of $x^3$ has length $q(q^2-1)$.
Using the two ramification points of the expression $x^3+x^2$ it follows at once that it has trivial stabilizer, and hence orbit length $q^2(q^2-1)^2$.
Finally, one easily finds that each of the expressions
$x^3+x$, or $x^3+\sigma x$
has stabilizer of order $2q$,
and hence orbit length
$q(q^2-1)^2/2$.
We conclude by difference that, over $\F_q$ of characteristic three, the number of cubic expressions with precisely four ramification points is given by
$q^2(q^2-1)^2(q-1)$, the same expression we found for larger characteristic.

Now consider a cubic expression $R(x)$ over $\F_q$ with four distinct ramification points
in $\barFq\cup\{\infty\}$.
Its stabilizer $H$ in $G(\F_q)$ permutes those ramification points, and the four branch points correspondingly.
Suppose $|H|>4$.
Because $H$ is a subgroup of the stabilizer of $R(x)$ in $G(\barFq)$,
Lemma~\ref{lemma:stabilizer} implies that $\F_q$ does not have characteristic three, that $H\cong A_4$,
and that $R(x)$ is equivalent to
$3x^2/(2x^3+1)$ over $\barFq$.
We will show that they are actually equivalent over $\F_q$, and that $q\equiv 1\pmod{3}$.

The Galois automorphism $\alpha\mapsto\alpha^q$ of $\barFq$ permutes the four ramification points of $R(x)$.
However, because $H$ is a subgroup of $G(\F_q)$ the Galois automorphism commutes with the action of $H$ on the ramification points.
Since $A_4$, the permutation group induced by $H$ on the ramification points, has trivial centralizer in the symmetric group $S_4$,
the Galois automorphism must act trivially on the ramification points.
It follows that the ramification points are in $\F_q\cup\{\infty\}$, and hence so are the corresponding branch points.
Using a M\"obius transformation which maps the four ramification points of $R(x)$
to those of $3x^2/(2x^3+1)$, which are $0$, $1$, and the the primitive sixth roots of unity,
and another suitable M\"obius transformation for the branch points,
we find that $R(x)$ is equivalent to $3x^2/(2x^3+1)$.
Furthermore, because the cross-ratio of the ramification points is a primitive sixth root of unity, that belongs to $\F_q$, and hence $q\equiv 1\pmod{3}$.


%
%

Summarizing, each orbit of cubic expressions with four ramification points has length at least $q^2(q^2-1)^2/4$,
except for precisely one orbit of length $q^2(q^2-1)^2/12$ if $q\equiv 1\pmod{3}$.
It follows that the number of orbits consisting of cubic rational expressions with four ramification points does not exceed $4(q-1)$,
including in the case where $q\equiv 1\pmod{3}$
(as the one shorter orbit only increases the integer bound by $1/3$).
The conclusion follows after including the remaining four orbits in the count.
\end{proof}


\section{Cubic expressions in characteristic two}\label{sec:small_char}

In this section we determine representatives and lengths for all $G(\F_q)$-orbits of cubic expressions
in characteristic two.
Recall from Theorem~\ref{thm:cubic_expr_alg_closed_2} that no such expression can have more than two ramification points.

\begin{theorem}\label{thm:cubic_expr_finite_2}
Let $\F_q$ be a finite field of characteristic two.
Let $\sigma\in\F_q$ be an element of absolute trace $1$.
If $q$ is a square then let $\theta$ be a non-cube in $\F_q$.
Then any cubic rational expression $R(x)$ over $\F_q$  is equivalent to precisely one of the following:
\begin{itemize}
\item[(i)]
$x^3$, with ramification indices $(3,3)$;
\item[(ii)]
$\displaystyle
\frac{x^3+\sigma x+\sigma}{x^2+x+\sigma+1}$,
with ramification indices $(3,3)$;
\item[(iii)]
$x^3 + x^2$,
with ramification indices $(3,2)$;
\item[(iv)]
$(x^3+1)/x$
or, only in case $q$ is a square,
$(x^3+\theta)/x$ or $(x^3+\theta^2)/x$;
all these have ramification indices $(2)$;
\item[(v)]
$(x^3 + c)/(x+c)$,
with ramification indices $(2,2)$,
for a unique $c\in\F_q\setminus\F_2$;
\item[(vi)]
$\displaystyle
\frac{x^3 +bx^2 +\sigma x +(b+1)\sigma}
{x^2 +x +b+1+\sigma}$,
with ramification indices $(2,2)$,
for a unique $b\in\F_q\setminus\F_2$.
\end{itemize}
\end{theorem}

\begin{proof}
According to Theorem~\ref{thm:cubic_expr_alg_closed_2} and Remark~\ref{rem:cubic_expr_alg_closed_2},
over the algebraic closure $\barFq$ of $\F_q$ the expression $R(x)$ is equivalent to either
$x^3$, or
$x^3+x^2$, or
$(x^3+1)/x$, or
$(x^3+c)/(x+c)$, with $c\in\barFq\setminus\F_2$.
Let $\tau\in\F_{q^2}$ satisfy $\tau^2+\tau=\sigma$.
Since
$\sigma+\sigma^2+\sigma^4+\cdots+\sigma^{q/2}=1$
by our choice of $\sigma$, we have
\[
\tau^q+\tau
=
(\tau+\tau^2)+(\tau+\tau^2)^2+\cdots+(\tau+\tau^2)^{q/2}
=1
\]
and, consequently,
$\tau^{q+1}=(\tau+1)\tau=\sigma$.

The simplest case is when $R(x)$ is equivalent to $x^3+x^2$ over $\barFq$, which occurs precisely when it has ramification indices $(3,2)$.
Because the ramification indices are different, each of the two ramification points must be fixed by the Galois automorphism $\alpha\mapsto\alpha^q$,
hence it belongs to $\Proj^1(\F_q)$, and so does the corresponding branch point.
Consequently, the ramification points can be assumed to be $\infty$ and $0$ and to coincide with the corresponding branch points,
yielding that $R(x)$ is equivalent to $x^3+x^2$ over $\F_q$.

Now suppose that $R(x)$ is equivalent to $x^3$ over $\barFq$, which occurs precisely when it has ramification indices $(3,3)$.
In this case the two ramification points could be either fixed or interchanged by the Galois automorphism $\alpha\mapsto\alpha^q$.
In fact, a cubic expression equivalent to $x^3$ over $\F_{q^2}$ but not over $\F_q$ is
\[
\frac{\tau x+\tau^q}{x+1}
\circ
x^3
\circ
\frac{x+\tau^q}{x+\tau}
=\frac{x^3+\sigma x+\sigma}{x^2+x+\sigma+1},
\]
which has $\tau$ and $\tau^q=\tau+1$ as ramification points, and also as the corresponding branch points.
An argument which we fully spelled out in the proof of Theorem~\ref{thm:x^r_finite} now shows that
$R(x)$ is either equivalent to $x^3$ over $\barFq$, or to the other expression found above.

If $R(x)$ is equivalent to $(x^3+1)/x$ over $\barFq$ then it has ramification indices $(2)$.
Hence its only ramification point is Galois-invariant and can be taken to be $\infty$, with $\infty$ as the corresponding branch point.
The further preimage of $\infty$ is also Galois-invariant and hence can be taken to be $0$, that is, $R(0)=\infty$.
As in the proof of Theorem~\ref{thm:cubic_expr_alg_closed_2},
we find that $R(x)$ is equivalent to $(x^3+c)/x$ for some $c\in\F_q^{\ast}$.
However, after those stipulations all freedom we have left is passing to
$a^2x\circ(x^3+c)/x\circ x/a=(x^3+a^3c)/x$.
If $q$ is not a square then $q\equiv -1\pmod{3}$, and so the cubing map is bijective on $\F_q$ and $R(x)$ is equivalent to $(x^3+1)/x$.
If $q$ is a square then $c$ equals a cube in $\F_q^\ast$ times either $1$, $\theta$, or $\theta^2$,
and hence $R(x)$ is equivalent to precisely one of the three expressions given.

Finally, suppose that $R(x)$ has ramification indices $(2,2)$.
Then according to Theorem~\ref{thm:cubic_expr_alg_closed_2} and Remark~\ref{rem:cubic_expr_alg_closed_2} it is equivalent to $(x^3+c)/(x+c)$,
over $\barFq$, for some $c\in\barFq\setminus\F_2$.
Uniqueness of $c$ implies its Galois-invariance, whence $c\in\F_q\setminus\F_2$.
Recall that such cubic expression has $\infty$ and $1$ as ramification points, coinciding with the corresponding branch points,
whose further preimages are $c$ and $0$.

Now if the ramification points of $R(x)$ belong to $\F_q$ then so do the corresponding branch points
and also their further preimages.
In this case the proof of Theorem~\ref{thm:cubic_expr_alg_closed_2} shows that $R(x)$ is actually equivalent to $(x^3+ax^2)/(x+1)$ over $\F_q$,
and in turn to $(x^3+c)/(x+c)$ according to Remark~\ref{rem:cubic_expr_alg_closed_2}.
This last expression has $\infty$ and $1$ as ramification points of index $2$, each coinciding with the corresponding branch point,
and the further preimages of the branch points are $c$ and $0$, respectively.

The other possibility is that the Galois automorphism $\alpha\mapsto\alpha^q$ interchanges the two ramification points of $R(x)$,
and consequently the corresponding branch points, and also their further preimages.
Setting $c=1/b^2$
we find
\[
B(x)\circ\frac{b^2x^3+1}{b^2x+1}\circ A(X)
=
\frac{x^3 +bx^2 +\sigma x +(b+1)\sigma}
{x^2 +x +b+1+\sigma},
\]
where
\[
A(x)=\frac{x+\tau^q(b+1)+b\tau}{bx+b\tau}
\qquad\text{and}\qquad
B(x)=\frac{b\tau x+\tau^q(b+1)+b\tau}{bx+1}.
\]
This cubic expression is equivalent to $(x^3+c)/(x+c)$ over $\F_{q^2}$ but not over $\F_q$, as it has ramification points $\tau$ and $\tau^q=\tau+1$.
These coincide with the corresponding branch points, whose further preimages are $\tau+b$ and $\tau^q+b$.
A similar argument as that in the proof of Theorem~\ref{thm:x^r_finite} now shows that
$R(x)$ is equivalent to this expression over $\F_q$.
\end{proof}

Note that when $q=2$ the cases (v) and (vi) of Theorem~\ref{thm:cubic_expr_finite_2} are vacuous,
hence any cubic expression over $\F_2$ is equivalent to one of the following four possibilities:
$x^3$, $(x^3+x+1)/(x^2+x)$, $x^3+x^2$, and $(x^2+1)/x$.
In general Theorem~\ref{thm:cubic_expr_finite_2} yields the following count of equivalence classes.

\begin{cor}\label{cor:orbit_count_char_2}
Over a finite field $\F_q$ of characteristic two there are exactly
$2q+2$ equivalence classes of cubic rational expressions if $q$ is a square, and $2q$ otherwise.
\end{cor}

\begin{theorem}\label{thm:cubic_expr_finite_2_count}
Assume the hypotheses and notation of Theorem~\ref{thm:cubic_expr_finite_2}.
Of all the $q^5(q^2-1)$ distinct cubic rational expressions over $\F_q$, precisely
\begin{itemize}
\item[(i)]
$q^2(q^2-1)(q+1)/2$ are equivalent to $x^3$;
\item[(ii)]
$q^2(q^2-1)(q-1)/2$ are equivalent to
$\displaystyle
\frac{x^3+\sigma x+\sigma}{x^2+x+\sigma+1}$;
\item[(iii)]
$q^2(q^2-1)^2$ are equivalent to
$x^3 + x^2$;
\item[(iv)]
$q^2(q^2-1)^2$ are equivalent to one of
$(x^3+1)/x$,
$(x^3+\theta)/x$ or $(x^3+\theta^2)/x$,
equally split among these in case $q$ is a square;
\item[(v)]
$q^2(q^2-1)^2/2$ are equivalent to
$(x^3+c)/(x+c)$,
for each given $c\in\F_q\setminus\F_2$;
\item[(vi)]
$q^2(q^2-1)^2/2$ are equivalent to
$\displaystyle
\frac{x^3 +bx^2 +\sigma x +(b+1)\sigma}
{x^2 +x +b+1+\sigma}$,
for each given $b\in\F_q\setminus\F_2$.
\end{itemize}
\end{theorem}

Adding up these numbers, allowing for $b,c$ to vary as appropriate, gives the expected total of $q^5(q^2-1)$.

\begin{proof}
\quad(i)\quad
Each element of the stabilizer of $x^3$, that is, a pair $(B,A)\in G(\F_q)$ such that $B(x)\circ x^3\circ A^{-1}(x)$,
must either fix or interchange the two ramification points of $x^3$, and similarly for the branch points.
This fact and a simple calculation show that the stabilizer consists of the pairs $(B,A)=(a^3x,ax)$
and $(a^3/x,a/x)$, with $a\in\F_q^\ast$.
Consequently, the stabilizer has order $2(q-1)$, whence the orbit of $x^3$ has length
$|G(\F_q)|/\bigl(2(q-1)\bigr)=q^2(q^2-1)(q+1)/2$.

\quad(ii)\quad
The expression $R(x)$ under consideration is equivalent to $x^3$ over $\F_{q^2}$, and hence its stabilizer $T_R$ in $G(\F_{q^2})$ is conjugated to
the one described in (i), but with $a$ ranging in $\F_{q^2}^{\ast}$, call that $T_{x^3}$.
With notation as in the proof of Theorem~\ref{thm:cubic_expr_finite_2} we have
\[
T_R=
\left(
\frac{\tau x+\tau^q}{x+1},
\frac{\tau x+\tau^q}{x+1}
\right)
\cdot
T_{x^3}
\cdot
\left(
\frac{x+\tau^q}{x+\tau},
\frac{x+\tau^q}{x+\tau}
\right),
\]
from which one may work out explicit expressions for all elements of $T_R$
and then decide which of them belong to $G(\F_q)$.
The second component of the element of $T_R$ conjugated to $(B,A)=(a^3x,ax)\in T_{x^3}$ is
\[
\frac{\tau x+\tau^q}{x+1}
\circ
ax
\circ
\frac{x+\tau^q}{x+\tau}
=
\frac{(a\tau+\tau^q)x+(a+1)\tau^{q+1}}
{(a+1)x+(a\tau^q+\tau)},
\]
and we see that the $q$th power automorphism maps this to the corresponding conjugate of $a^{-q}x$.
An analogous assertion clearly holds for the first component,
and a similar calculation would yield the same conclusion when taking $(B,A)=(a^3/x,a/x)\in T_{x^3}$.
Consequently, the elements of $T_R$ which are fixed by the $q$th power automorphism are precisely those conjugated to those in $T_{x^3}$ satisfying $a^{-q}=a$.
Therefore $|T_R\cap G_(\F_q)|=2(q+1)$, and so the orbit of $R(x)$ has length
$|G(\F_q)|/\bigl(2(q+1)\bigr)=q^2(q^2-1)(q-1)/2$.

\quad(iii)\quad
Because the two ramification points $\infty$ and $0$ of $x^3+x^2$ have different indices,
each must be fixed by its stabilizer of $R(x)$, and similarly the branch points.
Therefore, any element in the stabilizer has the form $(B,A)=(bx,ax)$ for some $a,b\in\F_q$,
and one concludes at once that $a=b=1$.
Hence in this case the stabilizer is trivial, and the orbit has length
$|G(\F_q)|=q^2(q^2-1)^2$.

\quad(iv)\quad
The expression $(x^3+c)/x$ has the single ramification point $\infty$, which is also the branch point, whose other preimage is $0$.
Consequently, if $(B,A)$ is in its stabilizer then $A$ must fix each of $\infty$ and $0$, and $B$ must fix $\infty$.
A quick calculation then shows that the stabilizer consists of the pairs $(ax,x/a)$ with $a^3=1$.
Hence the orbit has length
$|G(\F_q)|/3=q^2(q^2-1)^2/3$ when $3$ divides $q-1$,
which occurs precisely when $q$ is a square, and
$q^2(q^2-1)^2$ otherwise.

\quad(v)\quad
The expression
$(x^3+c)/(x+c)$,
for some $c\in\F_q\setminus\F_2$,
has $\infty$ and $1$ as ramification points of index $2$, each coinciding with the corresponding branch point,
and that the further preimages of the branch points are $c$ and $0$, respectively.
Hence any element of its stabilizer in $G(\F_q)$ either fixes all these four points,
or interchanges $\infty$ with $1$ and $c$ with $0$.
Consequently, the stabilizer consists of the identity and the pair $\bigl((x+c)/(x+1),(x+c)/(x+1)\bigr)$,
and so the orbit has length
$|G(\F_q)|/2=q^2(q^2-1)^2/2$.

\quad(vi)\quad
This case is similar to the previous case.
The expression under consideration has
has $\tau$ and $\tau^q=\tau+1$ as ramification points of index $2$, each coinciding with the corresponding branch point,
and the further preimages of the branch points are $\tau+b$ and $\tau^q+b$, respectively.
Arguing as in the previous case one concludes that its stabilizer consists of the identity and the pair $(x+1,x+1)$,
and so the orbit has length
$|G(\F_q)|/2=q^2(q^2-1)^2/2$.
\end{proof}

According to Theorems~\ref{thm:cubic_expr_finite_2} and~\ref{thm:cubic_expr_finite_2_count},
a `random' cubic expressions over a finite field $\F_q$ of characteristic two
is equally likely to have one of the three ramification types $(3,2)$, $(2)$ or $(2,2)$,
with the remaining ramification type $(3,3)$ becoming less likely as $q$ grows.

\bibliography{References}
\end{document}